\documentclass[12pt,a4paper]{amsart}
\usepackage{amsmath,amssymb,amsthm,graphicx, color}

\calclayout

\newtheorem{proposition}{Proposition}
\newtheorem{conjecture}{Conjecture}
\newtheorem{theorem}{Theorem}

\theoremstyle{definition}

\newtheorem{example}{Example}
\newtheorem{remark}{Remark}

\DeclareMathOperator*{\argmin}{arg\,min}

\newcommand{\co}{\mathrm{co\,}}

\newcommand{\HH}{\mathcal{H}}
\newcommand{\N}{\mathbb{N}}
\newcommand{\R}{\mathbb{R}}

\title[A counterexample to De Pierro's conjecture]{A counterexample to De Pierro's conjecture on the convergence of under-relaxed cyclic projections}

\author{Roberto Cominetti}
\address[Roberto Cominetti]{Universidad Adolfo Ib\'a\~nez, Chile.}
\author{Vera Roshchina}
\address[Vera Roshchina]{University of New South Wales, RMIT University and Federation University Australia}
\author{Andrew Williamson}
\address[Andrew Williamson]{RMIT University, Australia}
\thanks{
Roberto Cominetti gratefully acknowledges partial support from FONDECYT 1171501 and N\'ucleo Milenio ICM/FIC RC130003 ``\emph{Informaci\'on y Coordinaci\'on en Redes}''.
}
\thanks{Vera Roshchina's research was supported by the Australian Research Council grant DE150100240 and by Enabling Capability Platform Information \& Systems (Engineering) of RMIT University.}

\begin{document}
\maketitle

\begin{abstract} The  convex feasibility problem consists in finding a point in the intersection of a finite family of  closed convex sets. 
When the intersection is empty, a best compromise is to search for a point that minimizes the sum of the squared 
distances to the sets. In 2001, de Pierro conjectured that the limit cycles generated by the 
$\varepsilon$-under-relaxed cyclic projection method  converge when $\varepsilon\downarrow 0$ towards a least squares solution.
While the conjecture has been confirmed under fairly general conditions, we show that it is false 
in general by constructing a system of three compact convex sets in $\R^3$ for which the  $\varepsilon$-under-relaxed cycles do not converge.

\medskip


\medskip


\end{abstract}

\section{Introduction}

The  convex feasibility problem consists in finding a point in the intersection of finitely many nonempty closed convex sets $C_1,\ldots,C_m$ in a Hilbert space $\HH$. A solution can be approximated by the method of cyclic projections which loops through this finite list of sets by iteratively projecting the current iterate onto the next set in a cyclic manner. Under mild conditions ---for instance, if one of the sets is bounded--- this process converges weakly either to a feasible point in the intersection of the sets, or to a limit cycle if this intersection is empty (see \cite{Gubin}). In the case of two sets $C_1,C_2\subseteq \HH$ the iteration reduces to von Neumann's alternating projection method ({\em cf.} \cite{vonNeuman}) which converges to a two-point cycle that solves the minimal distance problem
\begin{equation}\label{eq:varchartwo}
\min_{x_1\in C_1 \atop x_2\in C_2} \|x_1-x_2\|,\notag
\end{equation}
provided that the latter  has a solution. Such variational characterisation does not exist for three or more sets. It was shown in \cite{RobertoVarChar} that for $m\geq 3$ there is no function $\Phi:\HH^m\to \R$ such that for any collection of compact convex sets $C_1,C_2,\dots, C_m\subseteq \HH$ the limit cycles are precisely the solutions of the minimisation problem 
$$
\min_{x_i\in C_i} \Phi(x_1,x_2,\dots, x_m).
$$

This lack of variational characterization can be countered by considering an under-relaxed version of the cyclic projection method. It was conjectured in \cite{DePierro} that the corresponding limit cycles converge towards a solution of the least squares problem
\begin{equation}\label{eq:ls}
S=\argmin_{u\in \HH} \sum_{i=1}^md(u,C_i)^2.
\end{equation}
Here $d(u,C)=\min_{x\in C}\|x-u\|$ denotes the distance from $u\in\HH$ to the closed convex set $C$.  
We recall that the minimum in $d(u,C)$ is attained at a unique point $\Pi_{C}(u)$ which is the projection
of $u$ onto $C$, and which is characterised as follows: 
a point $w\in C$ is the projection of $u\in \HH$ onto $C$ if and only if\footnote{If $C = \co A$ is the convex hull of a set $A$ it suffices to check \eqref{eq:projchar} for $v\in A$.}
\begin{equation}\label{eq:projchar}
\langle v-w,u-w\rangle \leq 0 \qquad \forall v\in C.
\end{equation}

The under-relaxed cyclic projection method fixes a relaxation parameter $\varepsilon\in (0,1]$ and on each 
iteration an $\varepsilon$-step is taken towards the next projection, namely, given an initial point $u_0\in \HH$ for all $k\in \N\cup \{0\}$ we iterate as
\begin{equation}\label{alt-pi-epsilon}
\left\{\begin{array}{ccl}
u_{km+1} & = &u_{km}+\varepsilon(\Pi_{C_1}(u_{km}) - u_{km})\\  
u_{km+2} & = &u_{km+1}+\varepsilon(\Pi_{C_2}(u_{km+1}) - u_{km+1}),\\  
&\vdots & \\
u_{km+m} & =& u_{km+m-1}+\varepsilon(\Pi_{C_m}(u_{km+m-1}) - u_{km+m-1}).
\end{array}\right.
\end{equation}
Note that the standard cyclic projection method corresponds to the choice $\varepsilon=1$.

Let $\mathbf{u}_k^\varepsilon=(u_{km+1},u_{km+2}, \dots, u_{km+m})\in\HH^m$ be the  $m$-tuple generated on the $k$-th loop of the 
under-relaxed iteration \eqref{alt-pi-epsilon}. Under mild conditions these $m$-tuples converge weakly when $k\to\infty$ to an $\varepsilon$-cycle $\mathbf{u^\varepsilon} = (u_1^\varepsilon,u_2^\varepsilon,\dots,u_m^\varepsilon)\in \HH^m$ such that  
\begin{equation}\label{eq:definecycle}
\left\{\begin{array}{ccl}
u_1^\varepsilon & =& u_m^\varepsilon+ \varepsilon(\Pi_{C_1}(u_m^\varepsilon)-u_m^\varepsilon),\\  
u_2^\varepsilon & = &u_1^\varepsilon+ \varepsilon(\Pi_{C_2}(u_1^\varepsilon)-u_1^\varepsilon),\\  
&\vdots  &\\
u_m^\varepsilon & =& u_{m-1}^\varepsilon+ \varepsilon(\Pi_{C_m}(u_{m-1}^\varepsilon)-u_{m-1}^\varepsilon).
\end{array}\right.
\end{equation}
As a matter of fact, \cite[Propositions~1.1 and~1.3 and Corollary~1.3]{BruckReich} 
show that for any starting point $u_0$ the tuple $\mathbf{u}_k^\varepsilon$ 
converges weakly to an $\varepsilon$-cycle $\mathbf{u}^\varepsilon$  if and only if  the set of solutions to \eqref{eq:definecycle} is nonempty. 
Note that the solution to  \eqref{eq:definecycle} may not be unique, so that in general the limit cycle $\mathbf{u}^\varepsilon$ might depend on the initial point $u_0$. To illustrate this observation and to give intuition for the subsequent discussion, we consider a simple example.

\begin{example}\label{example1} Consider the following system of three sets: two line segments
$$
C_1  := \co \{(-2,2,1), (-2,2,-1)\}, \qquad C_2  := \co \{(2,2,1), (2,2,-1)\},
$$
and the cylinder
$$
C_3 := \{(x,y,z)\,|\, x^2+y^2\leq 1, |z|\leq 1 \}.
$$
The least square solution set is given by the vertical segment 
$$
S=\left\{\left(0,\frac{5}{3},z\right):|z|\leq 1\right\}.
$$
If we start from a point $u_0=(x_0,y_0,z_0)$ at height $z_0$, all the iterates remain in the plane $z=z_0$ and so does the limit 
cycle $\mathbf{u}^\varepsilon$ which therefore depends on $z_0$ (see Fig.~\ref{fig:heights}).
\begin{figure}[ht]
\includegraphics[width=0.45\textwidth]{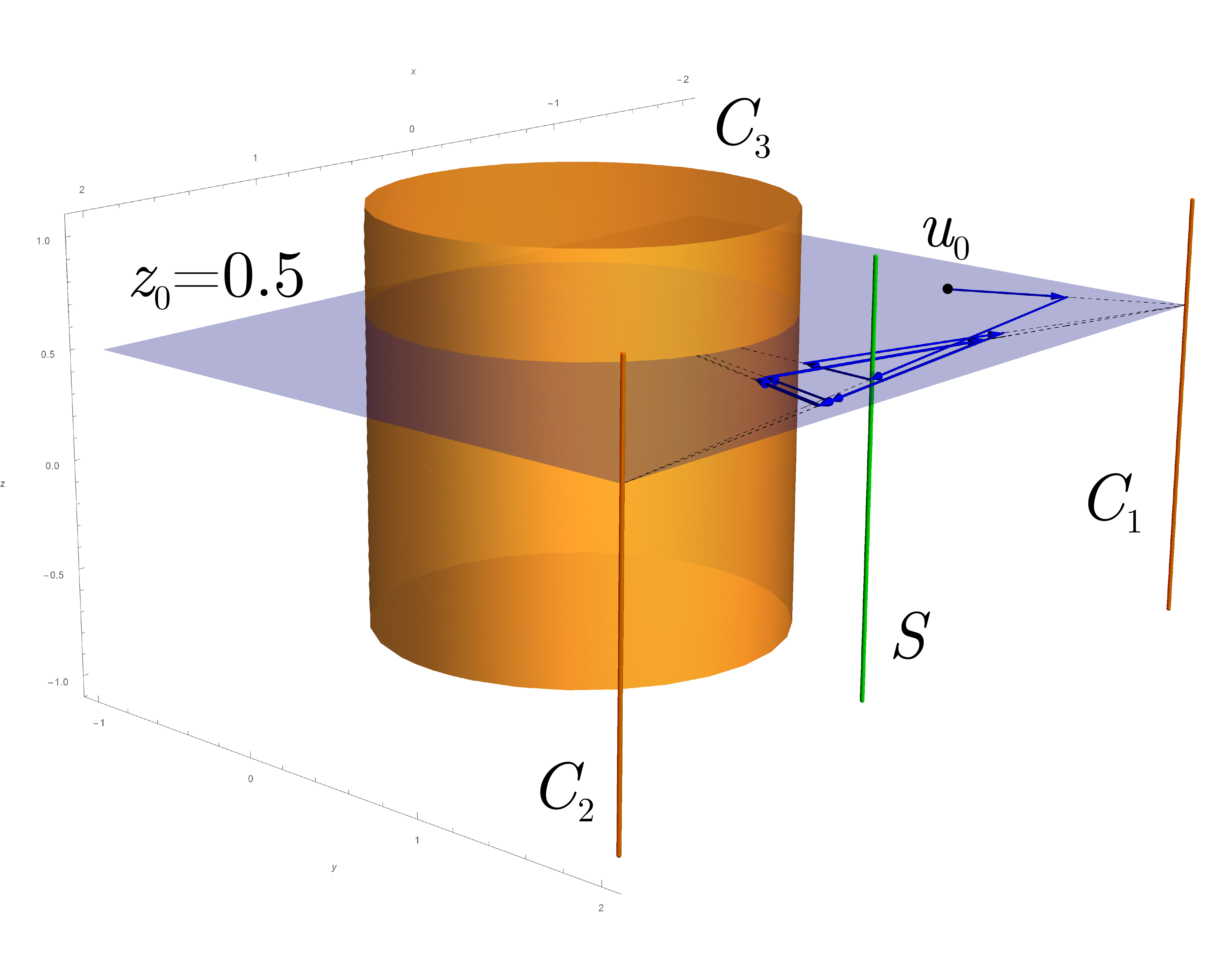}
\includegraphics[width=0.45\textwidth]{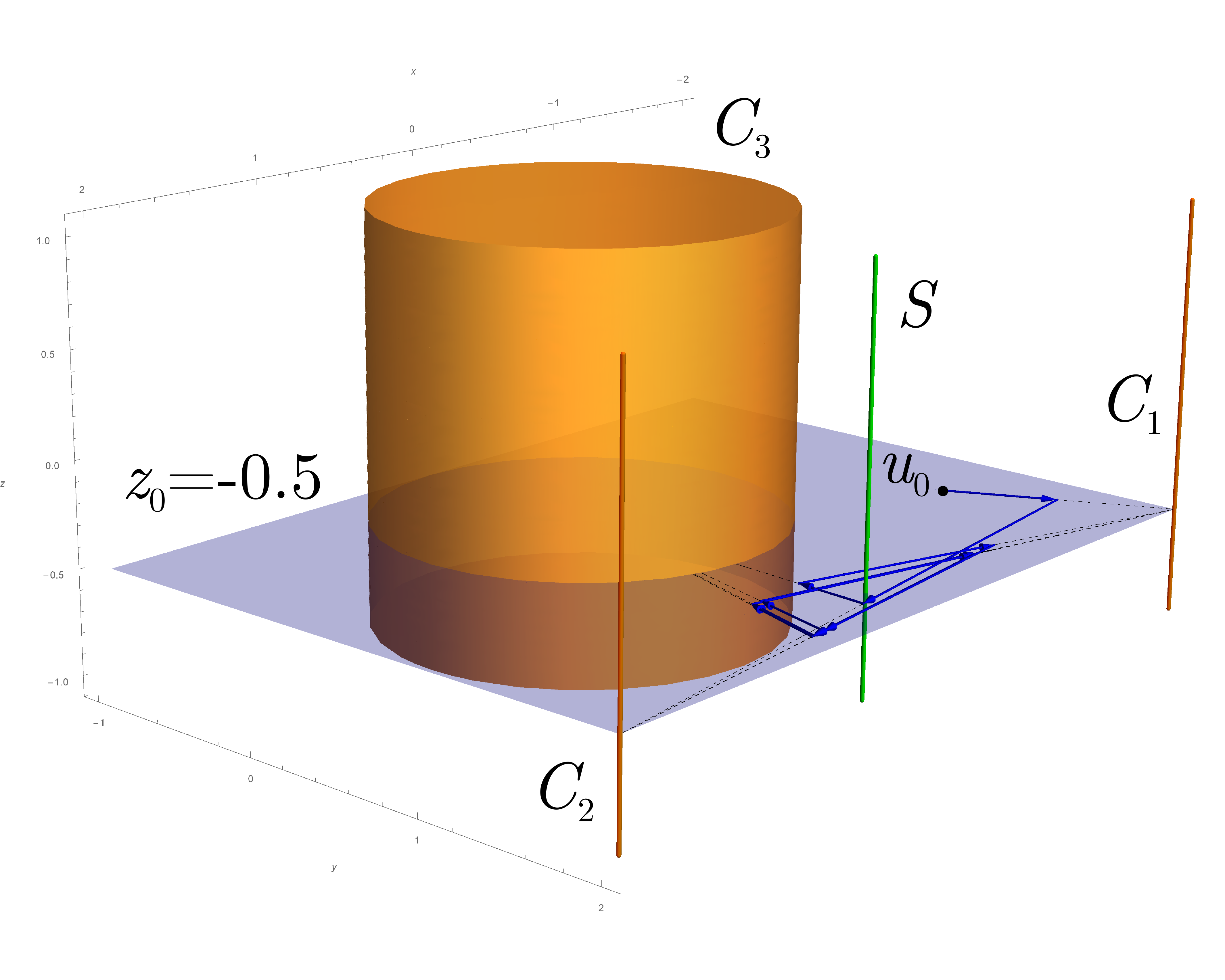}
\caption{The under-relaxed projections for $\varepsilon = \frac{1}{2}$ and different starting points in Example~\ref{example1}.}	
\label{fig:heights}
\end{figure}
Now, if we consider a fixed $u_0$ and we let $\varepsilon\downarrow 0$ the limit cycle shrinks 
towards the point in the least squares segment $S$ at height $z_0$.
Thus, the initial point  $u_0$ serves as an `anchor' that provides some 
hope for the limit cycles $\mathbf{u}^\varepsilon$ to converge as $\varepsilon\downarrow 0$. 

\begin{figure}[ht]
	\includegraphics[width=0.45\textwidth]{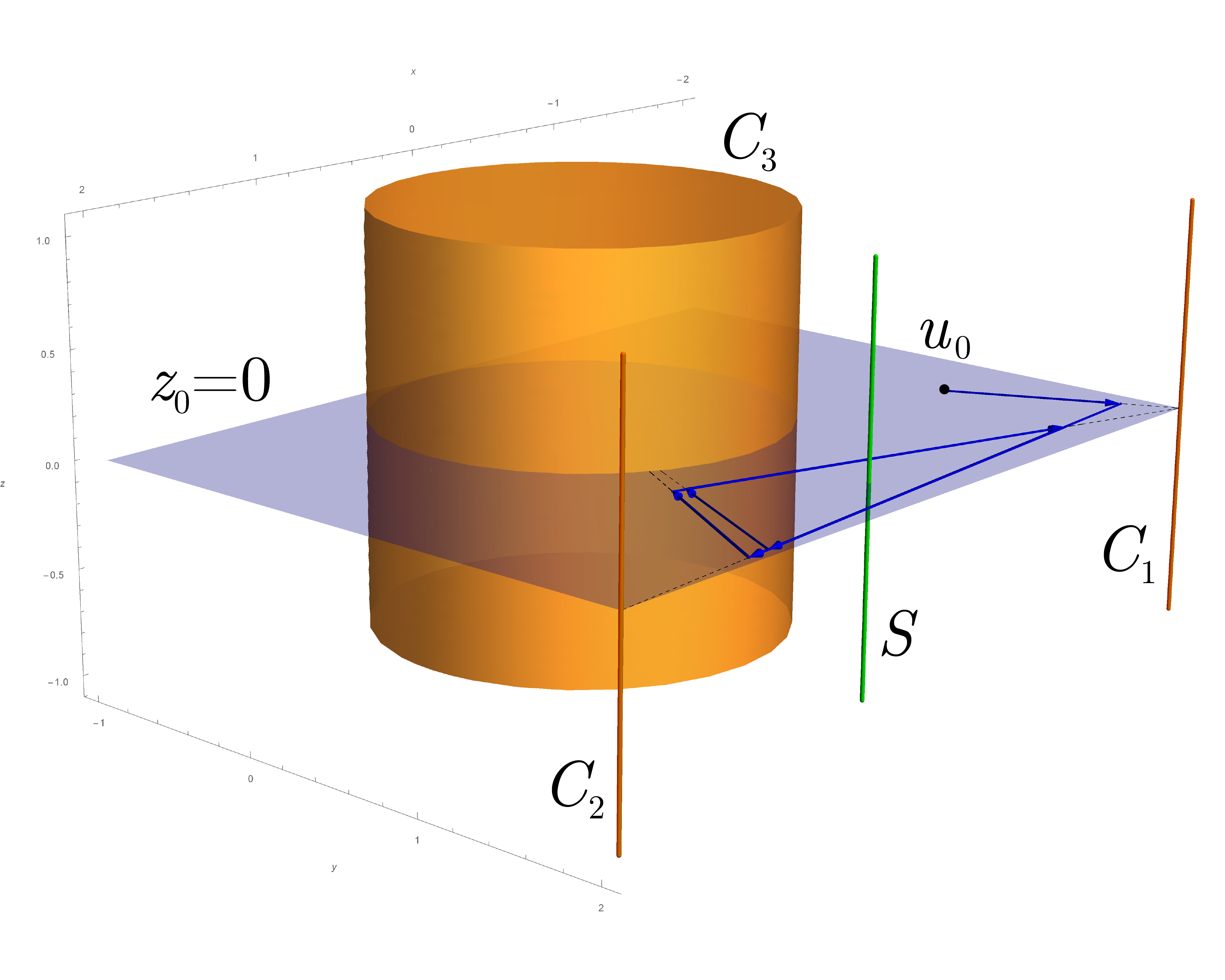}
	\includegraphics[width=0.45\textwidth]{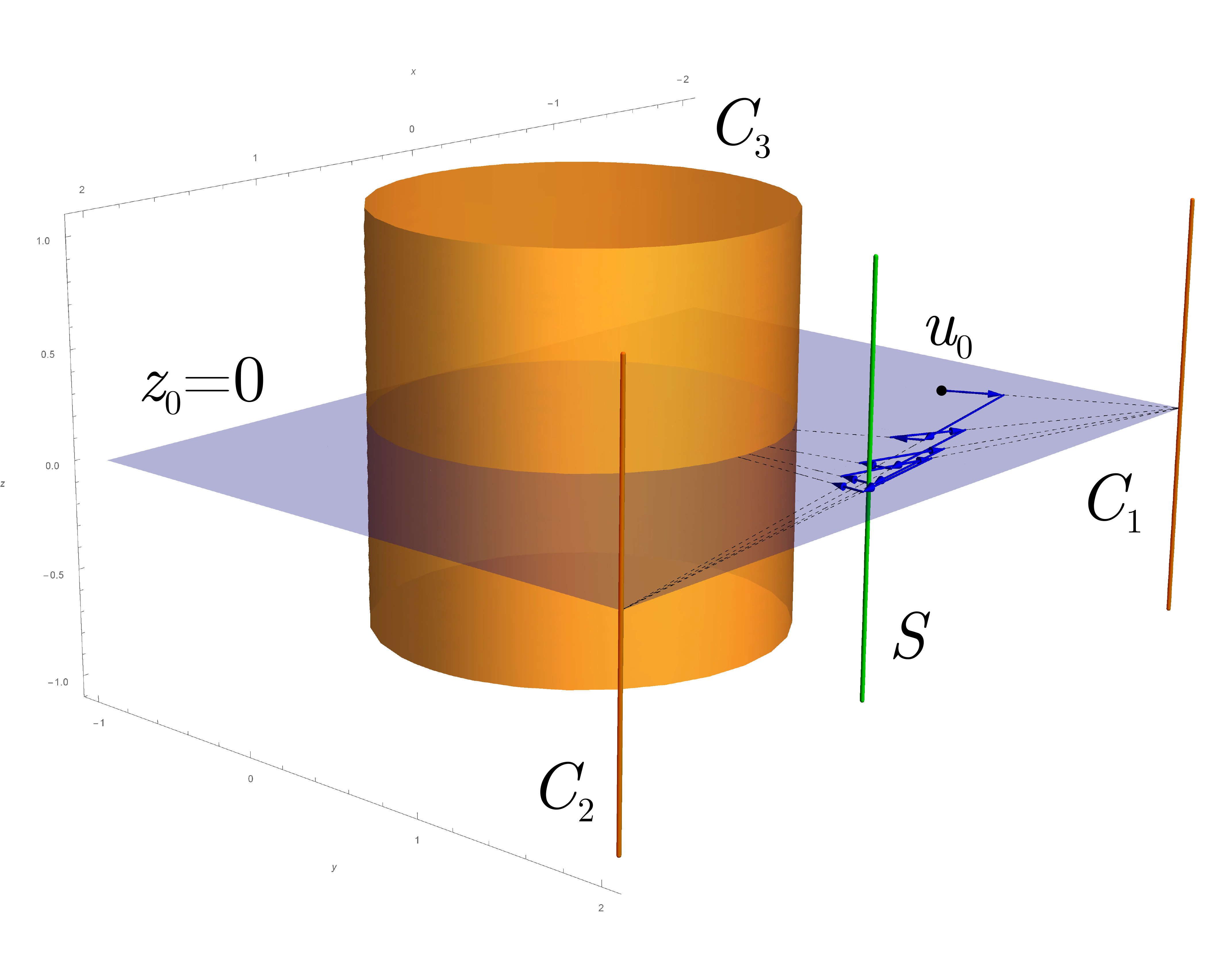}
	\caption{The iterative process for $\varepsilon = \frac{3}{4}$ and $\varepsilon = \frac{1}{4}$ and the same starting point in Example~\ref{example1}.}	
\label{fig:epsilons}
\end{figure}
\end{example}

Following \cite{DePierro}, we consider a fixed starting point $u_0$ and we focus on the existence of the limit as $\varepsilon\to 0$ for the 
corresponding limit cycle $\mathbf{u^\varepsilon}$, that is
\begin{equation}\label{eq:limitA}
\mathbf{u}=\lim_{\varepsilon \downarrow 0} \mathbf{u^\varepsilon}.
\end{equation}
Note that by letting $\varepsilon\to 0$ in \eqref{eq:definecycle} it readily follows that if the limit exists it must be of the form $\mathbf{u}=(\bar u,\bar u,\ldots,\bar u)$ for some $\bar u\in\HH$ (which again may depend on $u_0$).

Together with the limit cycles  $\mathbf{u^\varepsilon}$, the following iterative process was considered in \cite{DePierro}. 
Given a fixed sequence $\lambda_k\to 0$ with $\sum_k \lambda_k = +\infty$ define
$\mathbf{v}_k=(v_{km+1},v_{km+2},\ldots, v_{km+m})$ inductively by setting for $i=1,\dots, m$
$$
v_{km+i} = v_{km+i-1} + \lambda_{km+i} \left(\Pi_{C_m}(v_{km+i-1}) - v_{km+i-1}\right).
$$
In this case we are  interested in the existence of the limit
\begin{equation}\label{eq:limitB}
\mathbf{v} = \lim_{k\to \infty} \mathbf{v}_k
\end{equation}
which again is necessarily of the form $\mathbf{v}=(\bar v,\bar v,\ldots,\bar v)$ for some $\bar v\in\HH$.

In \cite{DePierro} it was conjectured that the convergence of the $\varepsilon$-cycles \eqref{eq:limitA} and of the modified process \eqref{eq:limitB} are  tied to the existence of solutions
to the least-squares problem \eqref{eq:ls}, namely 
\begin{conjecture}[\mbox{\cite[de Pierro]{DePierro}}] The least squares solution set $S$ is nonempty  if and only if for any starting point $u_0$ the limits \eqref{eq:limitA} and \eqref{eq:limitB} exist with both $\bar u$ and $\bar v$ in $S$.
\end{conjecture}

This conjecture has been confirmed under various conditions. In \cite{censor} it was proved for families of affine subspaces of $\R^n$, a result which was extended in \cite{BauschkeTranslates} to the infinite dimensional setting under a metric regularity condition. Beyond the case of affine subspaces, \cite[Theorem 2.8]{RobertoUnderRelaxed} describes several geometric conditions under which Conjecture 1 is true. The approach in \cite{RobertoUnderRelaxed} established a connection between the asymptotics of the  $\varepsilon$-cycles 
$\mathbf{u}^\varepsilon$ and the steepest descent trajectory $\dot u(t)=-\nabla\Phi(u(t))$
where $\Phi(x)=\frac{1}{2m}\sum_{i=1}^m d(x,C_i)^2$ is the least squares objective (up to the constant factor $\frac{1}{2m}$).
By exploiting this connection, \cite[Theorem 3.3]{RobertoUnderRelaxed} proves Conjecture 1 under a mild  geometrical condition.
In particular this condition holds automatically for the case of alternating projections where only $m=2$ sets are involved  \cite[Corollary 3.4]{RobertoUnderRelaxed}. 

All these known results require some additional condition so that it remains as an open question whether  the conjecture holds in full generality 
as stated by de Pierro. Our goal is to disprove the conjecture by constructing a system of three compact convex sets in $\R^3$ for which the limit \eqref{eq:limitA} 
does not exist. Our main result is as follows.

\begin{theorem}\label{thm:counterex} There exist compact convex sets $C_1$, $C_2$, $C_3$ in $\R^3$ such that for all $\varepsilon\in (0,1]$  
there is a unique $\varepsilon$-cycle $\mathbf{u}^\varepsilon$ satisfying \eqref{eq:definecycle}. Moreover $\mathbf{u}^\varepsilon$
diverges for $\varepsilon\to 0$ so that the limit \eqref{eq:limitA} does not exist.
\end{theorem}

Note that by compactness the least-squares problem \eqref{eq:ls} has a solution in this case. 
The counterexample is even more striking since the $\varepsilon$-cycle  $\mathbf{u}^\varepsilon$ is unique and therefore it is independent 
of the initial point: for each $u_0\in\HH$ the under-relaxed iteration \eqref{alt-pi-epsilon} converges for $k\to\infty$ 
towards this unique $\varepsilon$-cycle $\mathbf{u}^\varepsilon$, and de Pierro's conjecture fails.
Similarly the limit \eqref{eq:limitB} may fail to exist: to see this fix $\lambda_k\equiv\varepsilon$ for a large 
number of iterations so that $\mathbf{v}_k$ comes close to $\mathbf{u}^\varepsilon$, after which we shift to a smaller
$\lambda_k\equiv\varepsilon'$ again for a sufficiently large number of iterations so that $\mathbf{v}_k$ comes close to $\mathbf{u}^{\varepsilon'}$.
Proceeding in this manner for a suitable chosen sequence of $\varepsilon$'s we can force the full sequence $\mathbf{v}_k$ to oscillate between different cluster points of $\mathbf{u}^\varepsilon$.

The rest of this paper is structured as follows. Section~\S\ref{sec:counterexample} presents the counterexample.  In \S\ref{sec:reduction}
we discuss a two-dimensional reduction and we establish the equivalence between three- and two-dimensional cycles. In \S\ref{sec:twod} we study the properties of the two-dimensional cycles. Finally, the proof of Theorem \ref{thm:counterex} is presented in \S\ref{sec:proof}.

\section{The counterexample}\label{sec:counterexample}

Our counterexample is a variant of Example \ref{example1} consisting of the two line segments
$$C_1  := \co \{(-2,2,1), (-2,2,-1)\}, \qquad C_2  := \co \{(2,2,1), (2,2,-1)\},$$
as before, and the compact convex subset of the unit cylinder (see Fig. \ref{fig:sets3d})
$$C_3 := \overline{\co} \{p_k\, |\, k\in \N\},\qquad p_k = (\cos t_{k}, \sin t_{k}, (-1)^k)$$
where $\{t_k\}$ is a monotonically increasing sequence with $t_1 =\frac{\pi}{4}$ and $t_k\to \frac{\pi}{2}$ 
as $k\to \infty$.
\begin{figure}[ht]
\includegraphics[width=0.75\textwidth]{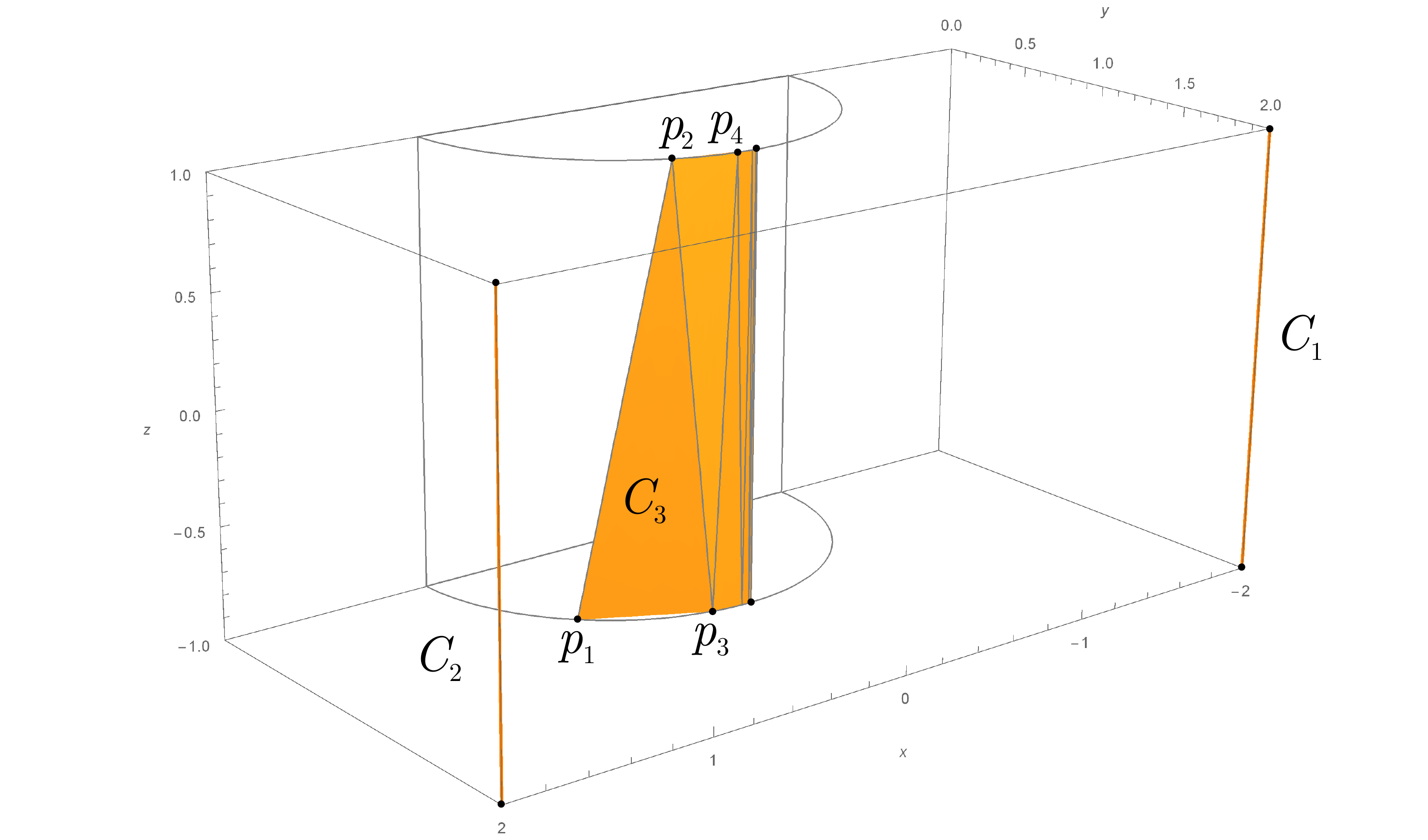}
\caption{The sets $C_1$, $C_2$ and $C_3$ for the sequence $t_k =
\frac{\pi}{2}(1-\frac{1}{2k})$.
}
\label{fig:sets3d}
\end{figure}

The least square solution set is again the segment $S=\{(0,\frac{5}{3},z):|z|\leq 1\}$. Indeed, observe that the 3-tuples $\{((-2,2,z),(2,2,z),(0,1,z))\,|\, |z|<1\} $ realise the relevant distances in \eqref{eq:ls} for the optimal least-squares solutions for Example~\ref{example1}. Since our new set $C_3$ is a subset of the original cylinder in Example~\ref{example1} and also contains the set $\{(0,1,z)\,|\, |z|\leq 1\}$, the least squares solutions are the same for these two problems.

Also, one can easily check that all the $p_k$'s are extreme points of $C_3$.
{\color{red} 
%
%
%
%
%
%
}
As will be seen in the sequel, the main feature of $C_3$ is the infinite sequence of facets $\co\!\{p_k,p_{k+1},p_{k+2}\}$ with alternating slopes which forces
the limit cycles $\mathbf{u}^\varepsilon$ to oscillate as $\varepsilon\to 0$ between the planes $z=-1$ and $z=1$,
``shadowing'' the zig-zag path $P$ that connects the points $p_1,p_2,p_3\ldots$

In the subsequent analysis we will consider the vertical projections onto the $xy$-plane.
Namely, for each $u=(x,y,z)\in\R^3$ we denote $u'=(x,y)$
and we consider the projected sets $C_1', C'_2, C'_3$. Letting $a= (-2,2)$, 
$b = (2,2)$ and $v_k=(\cos t_{k}, \sin t_{k})$, these projections are (see Fig.~\ref{fig:sets2d})
\begin{align}\label{eq:defCproj}
C'_1 & = \{(x,y)\,|\, (x,y,z) \in C_1 \} =  \{a\},\notag\\
C'_2 & = \{(x,y)\,|\, (x,y,z) \in C_2 \} =  \{b\},\\
C'_3 & = \{(x,y)\,|\, (x,y,z) \in C_3 \}  = \overline{\co}\left\{ v_k\, |\, k\in \N\right\}.\notag
\end{align}
\vspace{-0.5cm}

\begin{figure}[ht]
\includegraphics[width=0.7\textwidth]{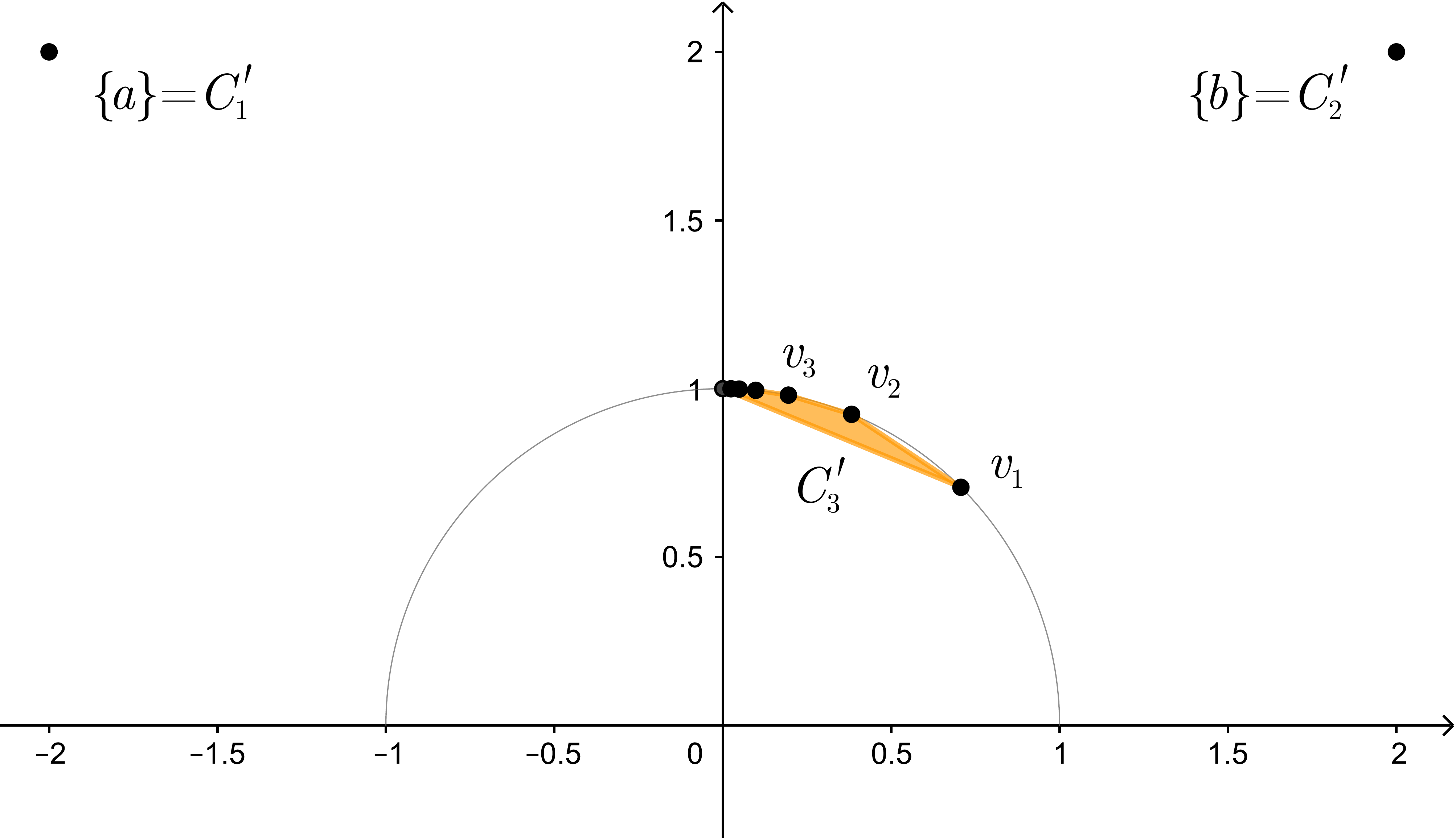}
\caption{The projections $C_1', C_2', C_3'$ of the sets $C_1, C_2, C_3$ onto the $xy$-plane.}
\label{fig:sets2d}
\end{figure}

Note that the zig-zag path $P$  projects vertically onto the path $P'$ within the $xy$-plane going through the points $v_1, v_2,v_3,\ldots$ 
which lies on the boundary of $C_3'$. Conversely for each point $p'\in P'$ the vertical line through $p'$ meets the set $C_3$
at a unique point $p$ which is in fact in $P$. We denote $z(p')$ the corresponding height so that $p=(p',z(p'))\in P$.

\section{Reduction to two dimensions}\label{sec:reduction}

Our first goal is to show that the vertical projection onto the $xy$-plane establishes a one-to-one correspondence between 
the three-dimensional  $\varepsilon$-cycles for $C_1,C_2,C_3$ and the two-dimensional $\varepsilon$-cycles for the projected sets
$C'_1,C'_2,C'_3$.

In the sequel we define the \emph{support} of an $\varepsilon$-cycle $(u_1,u_2, u_3)$ as the triple $(w_1,w_2,w_3)$
formed by the projections of $u_3$, $u_1$ and $u_2$ onto $C_1$, $C_2$ and $C_3$ respectively, that is
\begin{equation}\label{eq:proj}
w_1 := \Pi_{C_1}(u_3),\quad 
w_2 := \Pi_{C_2}(u_1),\quad 
w_3 := \Pi_{C_3}(u_2).
\end{equation}
Then, the system \eqref{eq:definecycle} can be written as
\begin{equation}\label{eq:proj2}
\begin{array}{ccl}
u_1 &=& (1\!-\!\varepsilon)u_{3} + \varepsilon w_1,\\  
u_2 &=& (1\!-\!\varepsilon)u_{1} + \varepsilon w_2, \\
u_3 &=& (1\!-\!\varepsilon)u_{2} + \varepsilon w_3,
\end{array} 
\end{equation}
from which it follows that the $u_i$'s can be recovered as convex combinations of the $w_i$'s 
\begin{equation}\label{eq:proj3}
\begin{array}{ccl}
u_1 &=& \frac{(1-\varepsilon)^2 w_2 +(1-\varepsilon)w_3+w_1}{\varepsilon^2-3\varepsilon+3}, \\  
u_2 &=& \frac{(1-\varepsilon)^2 w_3 + (1-\varepsilon)w_1+w_2}{\varepsilon^2-3\varepsilon+3}, \\
u_3 &=& \frac{(1-\varepsilon)^2 w_1 + (1-\varepsilon)w_2+w_3}{\varepsilon^2-3\varepsilon+3}.
\end{array} 
\end{equation}

\begin{proposition}\label{prop:DimRed} Let $C_1, C_2, C_3$ defined as in {\em \S\ref{sec:counterexample}}
 and $C'_1, C_2', C_3'$  their $xy$-projections. Then the triple $(u_1,  u_2,u_3)$ is an $\varepsilon$-cycle for $C_1, C_2, C_3$ with support  $(w_1,w_2,w_3)$ if and only if the following two properties hold:
\begin{itemize}
\item[(i)] the points $u_1$, $u_2$, $u_3$, $w_1$, $w_2$ and $w_3$ lie in a plane orthogonal to the $z$-axis; 
\item[(ii)] the projections  $(u_1',u_2',u_3')$  on the $xy$-plane are an $\varepsilon$-cycle for the two-dimensional  
sets $C_1', C_2', C_3'$ with support $(w_1',w_2',w_3')$.
\end{itemize}
\end{proposition}
\begin{proof} Fix some $\varepsilon>0$. We first prove the necessity. Let $(u_1,  u_2,u_3)$ be an $\varepsilon$-cycle for $C_1, C_2, C_3$. 
It is not difficult to see that any such cycle must lie in a plane orthogonal to the $z$-axis. Indeed, for any point $u=(x,y,z)$ with $|z|\leq 1$ we have 
$$
\Pi_{C_1}(u)=(-2,2,z)\quad\mbox{and}\quad \Pi_{C_2}(u)=(2,2,z). 
$$
It follows that $(w_1)_z=(u_3)_z$  and then the first equality in \eqref{eq:proj2} yields $(u_1)_z=(u_3)_z$.
Similarly $(w_2)_z=(u_1)_z$ so that the second equality in \eqref{eq:proj2} yields $(u_2)_z=(u_1)_z$.
Hence $(u_1)_z=(u_2)_z=(u_3)_z$ and the third equality in \eqref{eq:proj2} implies $(w_3)_z=(u_2)_z$.
Altogether
$$ (u_1)_z  = (u_2)_z=(u_3)_z=(w_1)_z = (w_2)_z= (w_3)_z$$
which proves (i).

In order to prove (ii) let us first note that for the projections on the $xy$-plane we have
$$
u_1' = u'_{3} + \varepsilon (w_1'-u_{3}'), \quad  
u_2' = u'_{1} + \varepsilon (w_2'-u_{1}'),  \quad
u_3' = u'_{2} + \varepsilon (w_3'-u_{2}'),  
$$
Moreover since the sets $C'_1$ and $C'_2$ are singletons, we clearly have 
$$
w_1' = a = \Pi_{C_1'}(u_3'), \quad w_2' = b = \Pi_{C_2'}(u_1'),
$$
so that it remains to show that $w_3' = \Pi_{C'_3}(u_2')$. Let $p'=\Pi_{C'_3}(u_2')$ and suppose by contradiction that $p'\neq w_3'$ so that 
$\|p'-u'_2\| <\|w_3'-u_2'\|$.
Then we have  
$$\|p'-u'_2\|^2  =  \|p'-w'_3\|^2 +\|w'_3-u'_2\|^2 + 2 \langle p'-w_3',w_3'-u'_2\rangle <\|w_3'-u'_2\|^2, $$
and therefore 
\begin{equation}\label{eq:pfprojineq}
2 \langle p'-w_3',u_2'-w_3'\rangle>\|p'-w'_3\|^2>0.
\end{equation}
Since $C'_3$ is the projection of $C_3$ on the $xy$-plane, there exists some point $p\in C_3$ such that $p=(p',z)\in C_3$ for $z \in [-1,1]$. 
Moreover, since $(u_2)_z=(w_3)_z$ it follows that
$$
\langle p- w_3, u_2-w_3\rangle  = \langle p' - w'_3, u'_2-w'_3\rangle >0
$$
which violates the property \eqref{eq:projchar} of projections and the fact that $w_3=\Pi_{C_3}(u_2)$. This contradiction 
implies that $p'=w_3'$ and establishes (ii).

Let us next prove the converse. Let $(u_1',u_2', u_3')$ be a two-dimensional cycle with support $(w_1',w_2', w_3')$.
The point $u_3'$ lies on the boundary of $C_3'$ and there is a unique height $z=z(u_3')\in [-1,1]$ for which the 
lifted point $u_3=(u_3',z)$ belongs to the zig-zag path $P\subseteq C_3$. We claim that the points $u_i=(u_i',z)$  for $i\in \{1,2,3\}$
constitute an $\varepsilon$-cycle for the sets $C_1,C_2,C_3$ with support $w_i=(w_i',z)$.
Indeed, we have
$$
u_1 = u_{3} + \varepsilon (w_1-u_{3}), \quad  
u_2 = u_{1} + \varepsilon (w_2-u_{1}),  \quad
u_3 = u_{2} + \varepsilon (w_3-u_{2}). 
$$
Clearly $w_1 = \Pi_{C_1}(u_3)$ and $w_2 = \Pi_{C_2}(u_1)$ so it suffices to show that $w_3 = \Pi_{C_3}(u_2)$. 
To this end we note that for all $p\in C_3$ and  its projection $p'$ on the $xy$-plane we have
\begin{equation}\label{eq:tech01}
\langle p-w_3,u_2-w_3\rangle  =  \langle p'-w'_3,u'_2-w'_3\rangle.
\end{equation}
Since $w'_3=\Pi_{C_3'}(u'_2)=\Pi_{C_3'}(u'_3) $ the expression in \eqref{eq:tech01} is non-positive and hence invoking \eqref{eq:projchar} we conclude that $w_3$ is the projection of $u_2$ onto $C_3$
as was to be proved. 
\end{proof}

Using the previous result, we may deduce the existence and uniqueness of an $\varepsilon$-cycle.
\begin{proposition}\label{prop:uniquecycle} For each $\varepsilon\in (0,1)$ there exists a unique $\varepsilon$-cycle
$(u_1,u_2,u_3)$ for $C_1,C_2,C_3$, and a unique $\varepsilon$-cycle $(u_1',u_2',u_3')$ for $C'_1,C'_2,C'_3$.
\end{proposition}

\begin{proof} In view of the one-to-one correspondence between the three- and two-dimensional $\varepsilon$-cycles established in Proposition \ref{prop:DimRed},
it suffices to consider the two-dimensional case. The existence of an $\varepsilon$-cycle $(u_1',u_2',u_3')$ follows from  the compactness of the sets 
$C'_1,C'_2,C'_3$ and general results in \cite{DePierro}.  In order to prove its uniqueness let us consider
 two $\varepsilon$-cycles $(u_1',u_2',u_3')$ and $(u_1'',u_2'',u_3'')$. Proceeding as in \eqref{eq:proj3} we get the following equalities in terms of their
 corresponding supports 
$$u_3'  = \frac{(1-\varepsilon)^2 w'_1 + (1-\varepsilon)w'_2+w'_3}{\varepsilon^2-3\varepsilon + 3}
\quad \text{and}\quad 
u''_3  = \frac{(1-\varepsilon)^2 w''_1 + (1-\varepsilon)w''_2+w''_3}{\varepsilon^2-3\varepsilon + 3},
$$
and since $w_1'=w_1''=a$ and  $w_2'=w''_2=b$ we obtain  
$$
(\varepsilon^2-3\varepsilon + 3)(u'_3-u''_3) = (w'_3-w''_3).
$$
Now, $w'_3$ and $w_3''$ are the projections of $u'_3$ and $u_3''$ onto $C_3$ so that
$\|w'_3-w''_3\|\leq \|u'_3-u''_3\|$, and since $\varepsilon^2-3\varepsilon + 3>1$ on $(0,1)$ 
it follows that $u_3'=u_3''$ and $w_3'=w_3''$. From this equality and property \eqref{eq:proj3} for the two-dimensional 
 cycles we readily get $u_1'=u_1''$ and $u_2'=u_2''$.
 \end{proof}

\section{Two-dimensional cycles}\label{sec:twod}

Let $(u_1',u_2',u_3')$ be an $\varepsilon$-cycle for the sets $C'_1, C'_2, C'_3$ with
support $(w_1',w_2',w_3')$. Proceeding as in \eqref{eq:proj3} we get
\begin{equation}
\label{eq:compcycles}
\begin{array}{ccl}
u'_1 & = &\frac{(1-\varepsilon)^2 w'_2 + (1-\varepsilon)w'_3+w'_1}{\varepsilon^2-3\varepsilon + 3},\\
u'_2 & = &\frac{(1-\varepsilon)^2 w'_3 + (1-\varepsilon)w'_1+w'_2}{\varepsilon^2-3\varepsilon + 3},\\
u'_3 & = &\frac{(1-\varepsilon)^2 w'_1 + (1-\varepsilon)w'_2+w'_3}{\varepsilon^2-3\varepsilon + 3}.
\end{array}
\end{equation}
In particular, since $w_1'=a$ and $w_2'=b$, the last equality implies
$$
\mbox{$\frac{\varepsilon^2-3\varepsilon + 3}{(1-\varepsilon)}$}(u'_3-w'_3) = (1-\varepsilon)(a-w'_3)+(b-w'_3).
$$
Since $w_3'=\Pi_{C_3}(u'_3)$ it follows that $u'_3-w_3'$ is a normal vector to $C'_3$ at $w_3'$. Choose any nonzero vector $d$ within the $xy$-plane orthogonal to $u'_3-w'_3$, so that 
\begin{equation}\label{eq:epsilonprelim}
0 = (1-\varepsilon) \langle a-w'_3,d\rangle +\langle b-w'_3,d\rangle.
\end{equation}
We note that $\langle a-w'_3,d\rangle \neq 0$ since otherwise $a-w'_3$ and $b-w'_3$ would be colinear, which is clearly impossible, and therefore 
from \eqref{eq:epsilonprelim} we get
$$
\varepsilon = 1+\frac{\langle b-w'_3,d\rangle}{\langle a -w'_3,d\rangle}.
$$
We use the intuition gained in the preceding discussion to prove the following key result. 

\begin{proposition}\label{prop:bijection} Let $k\in \N$ and set $d_k=\frac{v_{k+1}-v_k}{\|v_{k+1}-v_k\|}$. Then, for any point $c\in [v_k,v_{k+1})$ the triple $(w'_1,w'_2,w'_3)=(a,b,c)$
is the support of the $\varepsilon$-cycle $(u_1',u_2',u_3')$
corresponding to
$$\mbox{$\varepsilon = \varepsilon(c)=1+\frac{\langle b-c,d_k\rangle}{\langle a-c,d_k\rangle}\in(0,1)$}.$$
\end{proposition}
\begin{proof} Let us first show that  $\varepsilon(c)$ is well defined.
From Figure~\ref{fig:epsilon} 
\begin{figure}[ht]
\includegraphics[width=0.75\textwidth]{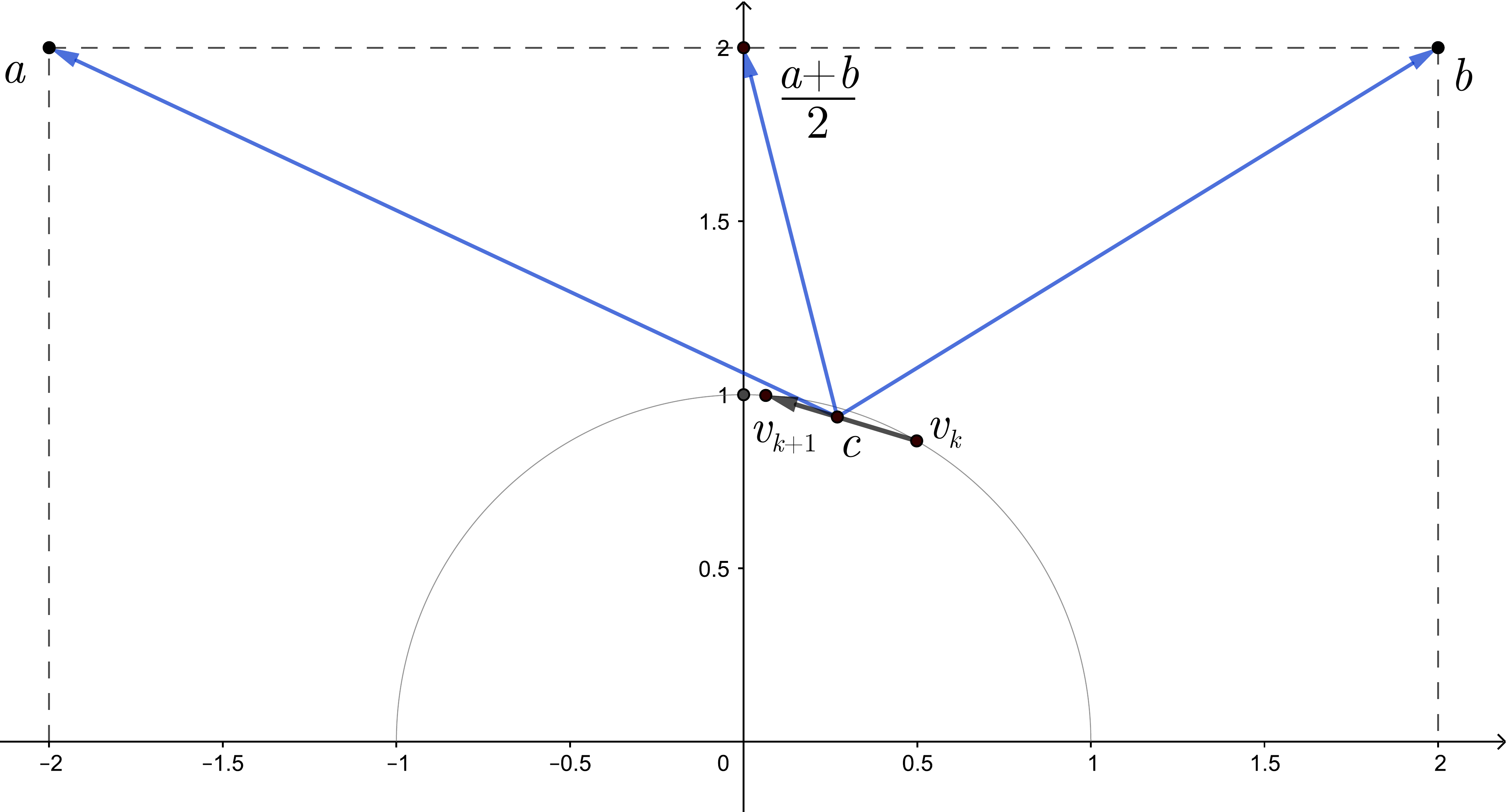}
\caption{Geometric argument for the proof of Proposition~\ref{prop:bijection}}
\label{fig:epsilon}
\end{figure}
we see that $a-c$ makes an acute angle with $d_k$ so 
it has a positive projection $\alpha(c)=\langle a-c,d_k\rangle>0$. Similarly $b-c$ has a negative projection 
$\beta(c)=\langle b-c, d_k\rangle<0$,  whereas $\alpha(c)+\beta(c)=2\langle\frac{a+b}{2}-c, d_k\rangle>0$.
Combining these facts it follows that $\varepsilon(c)=1+\frac{\beta(c)}{\alpha(c)}$ is well defined 
and belongs to $(0,1)$. 

Let us now consider $(u_1',u_2',u_3')$ computed from  \eqref{eq:compcycles} using $(w'_1,w'_2,w'_3)$ and $\varepsilon=\varepsilon(c)$.
In order to show that this is the (unique) $\varepsilon$-cycle with 
support $(w'_1,w'_2,w'_3)$ we note that  $w'_1 = a = \Pi_{C'_1}(u_3')$ and $w'_2 = b = \Pi_{C'_2}(u_1')$, so that 
it remains to prove $w'_3 = c= \Pi_{C'_3}(u_2')$. 
For the latter it suffices to check that $u_2'-c$ is orthogonal
to the segment $[v_k,v_{k+1})$ at the point $c$, which amounts to $\langle u_2'-c,d_k\rangle = 0$.
Now, the second equation in \eqref{eq:compcycles} yields
$$\mbox{$u'_2-c = \frac{ (1-\varepsilon)(a-c)+(b-c)}{\varepsilon^2-3\varepsilon + 3},$}$$
so that the result boils down to show  
$\langle (1-\varepsilon)(a-c)+(b-c), d_k\rangle = 0$.
This is equivalent to  $(1-\varepsilon) \alpha(c)+\beta(c)=0$ which follows directly from 
$\varepsilon = 1+\frac{\beta(c)}{\alpha(c)}$.
\end{proof}

\section{Proof of Theorem~\ref{thm:counterex}}\label{sec:proof}

We are ready to prove Theorem \ref{thm:counterex} using the three sets $C_1,C_2,C_3$ defined in \S\ref{sec:counterexample}.

From Proposition \ref{prop:uniquecycle} we know that for each $\varepsilon\in (0,1)$ there is a unique $\varepsilon$-cycle $\mathbf{u}^\varepsilon$.
Hence, for any starting point $u_0$ the under-relaxed iterates $\mathbf{u}_k^\varepsilon$ converge to this $\mathbf{u}^\varepsilon$.  
Theorem \ref{thm:counterex} will be proved if we show that $\mathbf{u}^\varepsilon$ oscillates as $\varepsilon\to 0$. 

Using Proposition \ref{prop:bijection} we see that for each point $c\in P'$ in the projection of the zig-zag path $P$
the triple $(w'_1,w'_2,w'_3)=(a,b,c)$ supports a 2-dimensional $\varepsilon$-cycle $(u_1',u_2',u_3')$ for $C_1',C_2',C_3'$ where 
$\varepsilon=\varepsilon(c)$. 
According to Proposition \ref{prop:DimRed} we may lift this $\varepsilon$-cycle to the height $z(c)$ to get an $\varepsilon$-cycle 
$(u_1,u_2,u_3)$ for $C_1,C_2,C_3$ with support $(w_1,w_2,w_3)$ where $u_i=(u_i',z(c))$ and $w_i=(w_i',z(c))$ for $i\in\{1,2,3\}$. 

Now, as  $c$ moves along the path $P'$ towards $v_\infty$ the lifted point $w_3=(c,z(c))$ moves accordingly along the zig-zag path $P$ with the 
height $z(c)$ oscillating between -1 and +1. It follows that the $\varepsilon(c)$-cycle $(u_1,u_2,u_3)$ also oscillates between height -1 and +1.

To complete the proof it remains to show that $\varepsilon(c)$ decreases to $0$ as $c$ moves along  $P'$ towards $v_\infty$.
Indeed, from Figure \ref{fig:epsilon} we see that when $c$ moves from $v_k$ to $v_{k+1}$ the projection $\alpha(c)$ decreases whereas $-\beta(c)$ increases 
so that their quotient $\frac{-\beta(c)}{\alpha(c)}$ increases and therefore $\varepsilon(c)$ decreases. 
Similarly, $\varepsilon(c)$ also decreases (with a jump discontinuity at $c=v_k$) as we pass from one segment $c\in [v_k,v_{k+1})$ 
to the next $c\in [v_{k+1},v_{k+2})$. Finally, since clearly $ d_k\to (-1,0)$ it follows easily that $\varepsilon(c)\downarrow 0$ as $c$
tends to $v_\infty$.

\begin{remark}
Note that as $\varepsilon\downarrow 0$ the $\varepsilon$-cycles $\mathbf{u}^\varepsilon=(u_1^\varepsilon,u_2^\varepsilon,u_3^\varepsilon)$
shrink so that $u_1^\varepsilon\approx u_2^\varepsilon\approx u_3^\varepsilon$. With a little more work one can show that the accumulation points of the cycles $\mathbf{u}^\varepsilon$ are precisely 
the triples of the form $(\bar u,\bar u,\bar u)$ with $\bar u$ a least square solution. Hence, the $\omega$-limit set
of $\mathbf{u}^\varepsilon$ as $\varepsilon\downarrow 0$ is the full least square solution set $S=\{(0,\frac{5}{3},z):|z|\leq 1\}$. 
\end{remark}

\begin{remark} The jump discontinuities of $\varepsilon(c)$ at $c=v_k$ correspond to ranges of $\varepsilon$ for which
the corresponding $\varepsilon$-cycle  remain at height 1 (for $k$ even) or height -1 (for $k$ odd) with $w_3'\equiv v_k$.
Note that although the support triple $(w_1,w_2,w_3)$ remains constant in these ranges, the corresponding 
$\varepsilon$-cycle  $(u_1,u_2,u_3)$ changes according to \eqref{eq:proj3}.
\end{remark}

\bibliographystyle{plain}
\bibliography{references}

\end{document}